\documentclass [12pt,a4paper]{amsart}
\usepackage[utf8]{inputenc}
\usepackage{amsfonts,amsmath,amsthm}
\usepackage{amssymb,latexsym}

\usepackage{mathrsfs}
\usepackage[hypcap]{caption}

\theoremstyle{plain}
\newtheorem{theorem}{Theorem}[section]

\newtheorem*{theorem*}{Theorem}
\newtheorem{lemma}[theorem]{Lemma}

\newtheorem{corollary}[theorem]{Corollary}
\newtheorem{proposition}[theorem]{Proposition}

\newcommand\vol{\mathrm{vol}}
\newcommand\tr{\mathrm{tr}}
\newcommand\sys{\mathrm{sys}}
\newcommand\genus{\mathrm{genus}}
\DeclareMathOperator{\SL}{SL}
\DeclareMathOperator{\N}{N}
\DeclareMathOperator{\GL}{GL}

\usepackage{hyperref}

\begin{document}

\title[On growth of systole along congruence coverings]{{\normalsize On growth of systole along congruence coverings of Hilbert modular varieties}}
\author{Plinio G. P. Murillo }

\thanks{\textit{Date:} \today}

\thanks{This work was supported by the Coordenação de Aperfeiçoamento de Pessoal de Nível Superior-CAPES}

\maketitle

\begin{abstract}
We study how the systole of principal congruence coverings of a Hilbert modular variety grows when the degree of the covering goes to infinity. We prove that given a Hilbert modular variety $M$ of real dimension $2n$, the sequence of principal congruence coverings $M_{I}$ eventually satisfies 
$$\sys\pi_{1}(M_{I})\geq \frac{4}{3\sqrt{n}}\log(\vol(M_{I}))-c,$$
where $c$ is a constant independent of $M_{I}$.
\end{abstract}

\section{\textbf{Introduction.}}
The {\it systole} of a Riemannian manifold is the least length of a non-contractible closed geodesic in $M$ and it is denoted by $\sys\pi_{1}(M)$. In 1994, P. Buser and P. Sarnak constructed in \cite{BS94} the first explicit examples of surfaces with systole growing logarithmically with the genus using a sequence of principal congruence coverings of an arithmetic compact Riemann surface. These sequences of surfaces $\lbrace S_{p}\rbrace$ satisfy the inequality
\begin{equation*}
\sys\pi_{1}(S_{p})\geq\frac{4}{3}\log(\genus(S_{p}))-c,
\end{equation*}

where $c$ is independent of $p$. These examples were generalized in 2007 by M. Katz, M. Schaps and U. Vishne to principal congruence coverings of any compact arithmetic Riemann surfaces and arithmetic hyperbolic 3-manifolds \cite{KSV07}. It is known that a sequence of principal congruence coverings of a compact arithmetic hyperbolic manifold attains asymptotically the logarithmic growth of the systole \cite[3.C.6]{Gro96} but the examples above are the only cases where the explicit constant in the systole growth was known so far. In particular, it would be interesting to understand how the asymptotic constant depends on the dimension.\\

The purpose of this paper is to generalize the construction of Buser and Sarnak in the context of Hilbert modular varieties which are non-compact Riemannian manifolds of dimensions $2n$. We will show that the sequence of principal congruence coverings $M_{I}\rightarrow M$ of a Hilbert modular variety eventually satisfies
\begin{equation}\label{lowerbound}
\sys\pi_{1}(M_{I})\geq \frac{4}{3\sqrt{n}}\log([\Gamma:\Gamma(I)]) -c
\end{equation}

where $c$ is a constant independent of $I$. We refer to Theorem \ref{maintheorem} for the precise statement of the main result.\\

Due to the non-compactness of $M$ it is a priori not clear if the systole of $M_{I}$ is bounded above by a logarithmic function of its volume. In fact, an interesting more general question is to understand if the systole of a sequence of congruence coverings of a non-compact finite volume arithmetic manifold of non-positive curvature and not flat grows logarithmically in its volume. An affirmative answer seems very plausible but, to our knowledge, it has not been established in the literature. In this regard we will prove that the sequence of principal congruence coverings $M_{I}\rightarrow M$ of a Hilbert modular variety eventually satisfies
\begin{equation}\label{upperbound}
\sys\pi_{1}(M_{I})\leq 4n^{3/2}\log([\Gamma:\Gamma(I)]).
\end{equation}

We expect that inequality \eqref{lowerbound} is asymptotically sharp while \eqref{upperbound} is not. Let us remark that these results give us the first examples of explicit constants for the growth of systole of a sequences of congruence coverings of arithmetic manifolds in dimensions greater than three.\\ 

We will begin in Section \ref{preliminaries} with recalling basic aspects of the action of $(\SL_{2}(\mathbb{R}))^{n}$ in $(\mathbb{H}^{2})^{n}$. We then define the congruences coverings of a Hilbert modular variety, and prove the systole upper bound \eqref{upperbound}. In Section \ref{mainpart} we estimate the length of closed geodesics of $M_{I}$ in terms on the norm of the ideal $I$, and in Section \ref{primecase} we relate the norm of the ideal $I$ with the index $[\Gamma:\Gamma(I)]$. All together this leads to the proof of inequality \eqref{lowerbound} in Section \ref{proofoftheorem}.\\

\textbf{Acknowledgements}. I would like to thank Mikhail Belolipetsky for his advice, support and many valuable suggestions on the preliminary versions of this work. I would also like to thank Cayo Doria for helpful discussions. 

\section{\textbf{Preliminaries}}\label{preliminaries}
\subsection{The action of $(\SL_{2}(\mathbb{R}))^{n}$ in $(\mathbb{H}^{2})^{n}$}

The group $\SL_{2}(\mathbb{R})$ acts in the upper half plane model of the hyperbolic plane $\mathbb{H}^{2}$ by fractional linear transformations: 

\begin{equation*}
Bz=\frac{az+b}{cz+d}, \hspace{2mm} \mbox{if} \hspace{2mm} B=\begin{pmatrix}
a & b \\
c & d	
\end{pmatrix}
\mbox{and}  \hspace{2mm} z\in\mathbb{H}^{2}.
\end{equation*}

An element $B\in \SL_{2}(\mathbb{R})$ is called \textit{elliptic} if has a fixed point in $\mathbb{H}^{2}$, \textit{parabolic} if has no fixed points in $\mathbb{H}^{2}$ and has only one fixed point in $\partial\mathbb{H}^{2}$ and \textit{hyperbolic} if it has no fixed points in $\mathbb{H}^{2}$ and has two fixed points in $\partial\mathbb{H}^{2}$. An equivalent description is the following:

\begin{flushleft}
$B$ is \textit{elliptic} if and only if $\lvert \tr(B) \rvert < 2$;\\
$B$ is \textit{parabolic} if and only if $\lvert \tr(B) \rvert = 2$;\\
$B$ is \textit{hyperbolic} if and only if $\lvert \tr(B) \rvert > 2$.\\
\end{flushleft} 

where $\tr(B)$ denotes the trace of the matrix $B$.\\

Given a hyperbolic transformation $B$, the \textit{translation length} of $B$, denoted by $\ell_{B}$, is defined by
$$\ell_{B}=\inf\lbrace d_{\mathbb{H}^{2}}(z,Bz);
z\in\mathbb{H}^{2}\rbrace.$$

This infimum is attained in the points on the only geodesic $\bar{\alpha}_{B}$ in $\mathbb{H}^{2}$ joining the fixed points of $B$ in $\partial\mathbb{H}^{2}$. The transformation $B$ fixes $\bar{\alpha}_{B}$ and acts on it as a translation. In particular, if a subgroup $\Gamma\subset \SL_{2}(\mathbb{R})$ acts properly discontinuously and freely on $\mathbb{H}^{2}$, every hyperbolic element $B\in\Gamma$ determines a non-contractible closed geodesic $\alpha$ on the Riemann surface $\mathbb{H}^{2}/\Gamma$, whose length is equal to the translation length $\ell_{B}$ of $B$. Reciprocally, a closed geodesic $\alpha$ in $\mathbb{H}^{2}/\Gamma$ lifts to a geodesic $\bar{\alpha}_{B}$ in $\mathbb{H}^{2}$ fixed by a hyperbolic matrix $B\in\Gamma$.\\

On the other hand, since $B$ is hyperbolic, $B$ is conjugate to a matrix of the form 
$$\begin{pmatrix}
\lambda & 0 \\
0 & \lambda^{-1}	
\end{pmatrix},$$ 
where $\lvert\lambda\rvert=e^{\frac{\ell_{B}}{2}}$. Hence  $2\cosh(\frac{\ell_{B}}{2})=\lvert \tr(B)\rvert$ and for any $z\in\mathbb{H}^{2}$ we have

\begin{equation}\label{distance and trace}
d_{\mathbb{H}^{2}}(z, Bz)\geq 2\log(\lvert \tr(B)\rvert-1)>0. \hspace{3mm} 
\end{equation}
We refer to \cite[Chapter 7]{Beardon} for further details about the geometry of the isometries of the hyperbolic plane $\mathbb{H}^{2}$.\\

The action of $\SL_{2}(\mathbb{R})$ on $\mathbb{H}^{2}$ extends to an action of the $n$-fold product $(\SL_{2}(\mathbb{R}))^{n}$ on the $n$-fold product $(\mathbb{H}^{2})^{n}$ of $n$ hyperbolic planes in a natural way: If $z=(z_{1},\ldots,z_{n})\in (\mathbb{H}^{2})^{n}$ and $B=(B_{1},\ldots,B_{n})\in~(\SL_{2}(\mathbb{R}))^{n}$, 
$$Bz:=(B_{1}z_{1},\ldots,B_{n}z_{n}),$$

where the action in every factor is the action by fractional linear transformations.\\

Let us recall the definition of a Hilbert modular variety (see \cite{Freitag}). Let $K$ be a totally real number field of degree $n$, $\mathcal{O}$ the ring of integers of $K$ and $\sigma_{1},\ldots,\sigma_{n}$ denote the $n$ embeddings of $K$ into the real numbers $\mathbb{R}$. The group $\SL_{2}(\mathcal{O})$ becomes an  arithmetic non-cocompact irreducible lattice of the semi-simple Lie group $G=(\SL_{2}(\mathbb{R}))^{n}$ via the map $\Delta(B)=(\sigma_{1}(B),\ldots,\sigma_{n}(B)),$ where $\sigma_{i}(B)$ denotes the matrix obtained from the matrix $B$ applying $\sigma_{i}$ to the entries of $B$ (see \cite[Proposition 5.5.8]{Mor15}).
Via this embedding, $\SL_{2}(\mathcal{O})$ acts with finite covolume  on the $n$-fold 	product of hyperbolic planes $(\mathbb{H}^{2})^{n}$. The quotient $M=(\mathbb{H}^{2})^{n}/\SL_{2}(\mathcal{O})$ is called a \textit{Hilbert modular variety} and the group $\SL_{2}(\mathcal{O})$ is called a \textit{Hilbert modular group}.

\subsection{Congruence coverings of $M$}\label{Congruence subgroups of G}

If $I\subset\mathcal{O}$ is an ideal, the natural projection $\mathcal{O}\rightarrow \mathcal{O}/I$ induces a group homomorphism

$$\SL_{2}(\mathcal{O}) \xrightarrow{\pi_{I}} \SL_{2}(\mathcal{O}/I).$$

Let us denote by $\Gamma(I):=ker(\pi_{I})$ the \textit{principal congruence subgroup of $\Gamma$} associated to $I$. Since $\mathcal{O}/I$  is finite, $\Gamma(I)$ is a finite-index subgroup of $\Gamma$ for any ideal $I$ of $\mathcal{O}$. We associate to $\Gamma(I)$ a \textit{congruence cover} $M_{I}=(\mathbb{H}^{2})^{n}/\Gamma(I)\rightarrow M$. Note that $\SL_{2}(\mathcal{O})$ is an irreducible lattice in $(\SL_{2}(\mathbb{R}))^{n}$ and so the varieties $M$ and $M_{I}$ do not split into products. We remark that $M$ has quotient singularities, so the covering $M_{I}\rightarrow M$ should be interpreted in the orbifold sense. For large enough $I$ the varieties $M_{I}$ are manifolds (see Corollary \ref{riemannian structure}). \\

This construction is a particular case of a more general situation: If $G$ is a semi-simple Lie group, a discrete subgroup $\Gamma\subset G$ is called arithmetic if there exists a number field $k$, a
$k$-algebraic group H, and a surjective continuous homomorphism 
$\varphi: \mbox{H}(k \otimes_{\mathbb{Q}}\mathbb{R}) \rightarrow G$ with
compact kernel such that $\varphi(\mbox{H}(\mathcal{O}_{k}))$ is commensurable to $\Gamma$, where $\mbox{H}(\mathcal{O}_{k})$
denote the $\mathcal{O}_{k}$-points of H with respect to some fixed embedding of H into $\GL_{m}$. For any ideal $I\subset\mathcal{O}_{k}$ the principal congruence subgroup of $\mbox{H}(O_{k})$ associated to $I$ is defined by 
$$\mbox{H}(I):=ker( \mbox{H}(\mathcal{O}_{k})\xrightarrow{\pi_{I}} \mbox{H}(\mathcal{O}_{k}/{I}))$$
where $\pi_{I}$ is the reduction map modulo $I$. Any discrete subgroup of $G$ containing some of these subgroups $\mbox{H}(I)$ is called a \textit{congruence subgroup of G}.\\

By Margulis Arithmeticity Theorem (see \cite[Chapter 5]{Mor15} and the references therein), for $n\geq 2$ any irreducible lattice in $(\SL_{2}(\mathbb{R}))^{n}$ is  arithmetic. A Serre's conjecture, proved to be true in the non-uniform case, shows that any non-uniform lattice of $(\SL_{2}(\mathbb{R}))^{n}$ is a congruence subgroup. \\

The coverings $M_{I}\rightarrow M$ are regular coverings because the subgroups $\Gamma(I)$ are normal subgroups of $\Gamma$. It is worth to note that in a sequence of non-regular congruence coverings of an arithmetic manifold the systole could grow slower than logarithmically with respect to the volume (\textit{cf}~\cite[Section 4.1]{LG14}). 

\subsection{Upper bound for the systole growth of $M_{I}$}\label{closedgeodesics}

The problem of determining if a manifold has a closed geodesic is in general very difficult. However, in the case of manifolds being quotients of hyperbolic spaces it appears to be easier to understand.\\

As was explained above, if $\Gamma$ is any discrete group of isometries of $\mathbb{H}^{2}$ acting freely on $\mathbb{H}^{2}$, every hyperbolic element $\gamma\in\Gamma$ produces a non-contractible closed geodesic on $\mathbb{H}^{2}/\Gamma$. We can use this idea to see that the quotients $M_ {I}$ which we are interested in have closed geodesics.\\

Suppose $I\subset\mathcal{O}$ is an ideal with $\N(I)>2$ and such that $M_{I}$ is a Riemannian manifold (see Corollary \ref{riemannian structure}). The norm $\N(I)$ is a rational integer with $\N(I)\in I$, so if we take the matrix
$$B= \left(
\begin{array}{cc}
1-\N(I)^{2}&  \N(I) \\
-\N(I) & 1 \
\end{array}
\right),$$ 
then $B\in \Gamma(I)$ and $\lvert \tr(\sigma_{i}(B))\rvert > 2$ for any $i=1,\ldots,n$. This means that the matrices $\sigma_{1}(B)=\sigma_{2}(B)=\cdots=\sigma_{n}(B)$ are hyperbolic and if we take $\bar{\alpha}$ being the only geodesic in $\mathbb{H}^{2}$ fixed by $B$, the curve $\bar{\beta}=\bar{\alpha}\times\cdots\times\bar{\alpha}$ is a geodesic in $(\mathbb{H}^{2})^{n}$ fixed by $(\sigma_{1}(B),\ldots,\sigma_{n}(B))$, and $\bar{\beta}$ project to a non-contractible closed  geodesic $\beta$ in $M_{I}$. Note that this geodesic might not be the shortest one, so $\sys\pi_{1}({M_{I}})\leq \ell(\beta)=\sqrt{n}\ell_{B}$, where $\ell_{B}$ denotes the translation length of $B$ along $\bar{\alpha}$.\\

We know that $2\cosh(\frac{\ell_{B}}{2})=\lvert \tr(B)\rvert=\N(I)^{2}-2 < \N(I)^{2},$ and so 

\begin{equation*}
\sys\pi_{1}({M_{I}})\leq 4\sqrt{n}\log(\N(I)).
\end{equation*}

Now let $r$ be the minimum of the positive rational integers in the ideal $I$. Then $\N(I)^{\frac{1}{n}}\leq r \leq \N(I)$ because $\N(I)\in I$ and $r^{n}=\N(r)\geq \N(I)$. For any distinct positive rational integers $m,n\in\lbrace 1,\ldots, r\rbrace$ the matrices  
\[
 \left(
\begin{array}{cc}
1-m^{2}& m \\
-m & 1 \
\end{array}
\right)\quad
\mbox{and}\quad
\left(
\begin{array}{cc}
1-n^{2}& n \\
-n & 1 \
\end{array}
\right)
\]

are not equivalent in the quotient $\SL_{2}(\mathcal{O})/\Gamma (I)$, hence $[\Gamma:\Gamma(I)]\geq r \geq \N(I)^{\frac{1}{n}}$. With all this we conclude that 
$$\sys\pi_{1}({M_{I}})\leq 4n^{\frac{3}{2}}\log([\Gamma:\Gamma(I)]).$$

This proves inequality \ref{upperbound}.

\section{\textbf{Distance estimate for congruence subgroups.}}\label{mainpart}

In this section we will prove that the congruence subgroups $\Gamma(I)$ act freely on $(\mathbb{H}^{2})^{n}$ when the norm of the ideal $I$ is big enough and we will relate the length of closed geodesics in $M_{I}$ to the norm of the ideal $I$. The first fact follows from Selberg's Lemma \cite[§4.8]{Mor15}  but in our case the proof gives an explicit bound in terms of the norm of $I$. Some of the ideas are inspired by \cite{KSV07} where the authors study the systoles of compact arithmetic hyperbolic surfaces and 3-manifolds.\\

We would like to make some comments about the notation we will use in this section. Sometimes we are going to use the notation $A$ or \\ $(\sigma_{1}(A), \ldots,\sigma_{n}(A))$ indifferently for the same element in $\SL_{2}(\mathcal{O})$ or its image in $(\SL_{2}(\mathbb{R}))^{n}$ via the map $\Delta$ defined in Section \ref{preliminaries}.

Take 
$A=\left(
\begin{array}{cc}
a & b \\
c & d \
\end{array}
\right)\in \SL_{2}(\mathcal{O})$. For our purpose it is convenient to express $A$ in the form $$A=\left(
\begin{array}{cc}
x_{0}+x_{1} & x_{2}+x_{3} \\
x_{2}-x_{3} & x_{0}-x_{1} \
\end{array}
\right),$$  
where $x_{0}=\frac{a+d}{2}, x_{1}=\frac{a-d}{2}, x_{2}=\frac{b+c}{2}$ and $x_{3}=\frac{b-c}{2}$ are elements of the field $K$. We have $x_{0}^{2}-x_{1}^{2}-x_{2}^{2}+x_{3}^{2}=1$ and we write  $y_{0}=x_{0}-1$. \\

With these notations, if $I\subset\mathcal{O}$ is an ideal and $A\in\Gamma(I)$ then $2x_{0}-2\in I$ and $2x_{i}\in I$ for $i=1,2,3$. In terms of fractional ideals it means that $y_{0}, x_{1}, x_{2}$ and $x_{3}$ lie in $\frac{I}{2}$.

\begin{lemma} \label{firstlemma}
If $A\in\Gamma(I)$, then $y_{0}\in\frac{I^{2}}{8}$. In particular, if $y_{0}\neq 0$ then $\lvert\N(y_{0})\rvert\geq\frac{1}{8^{n}}\N(I)^{2}$.  
\end{lemma} 

\begin{proof}
We know that $A\in\Gamma(I)$ implies $x_{0}-1, x_{1},x_{2},x_{3}\in \frac{I}{2}$. Now replacing $x_{0}=1+y_{0}$ in the equation $x_{0}^{2}-x_{1}^{2}-x_{2}^{2}+x_{3}^{2}=1$ we obtain

 $$2y_{0}=-y_{0}^{2}+x_{1}^{2}+x_{2}^{2}-x_{3}^{2}\in \frac{I^{2}}{4}.$$

Hence $y_{0}\in\frac{I^{2}}{8}$.
\end{proof}

\begin{lemma}\label{secondlemma}
If $A\in\Gamma(I)$ with $y_{0}\neq 0$ then $|\tr(\sigma_{j}(A))|\geq \frac{\left(\N(I)\right)}{4}^{\frac{2}{n}}-2$ for some $j\in\{1,\dots,n\}.$
\end{lemma}

\begin{proof}
By the definition, $\N(y_{0})=\prod_{j=1}^{n}\sigma_{j}(y_{0})$, so by Lemma \ref{firstlemma}, for some $j\in\{1,\ldots,n\}$, we have $|\sigma_{j}(y_{0})|\geq  \frac{\N(I)^\frac{2}{n}}{8}$. Therefore\\ $|\tr(\sigma_{j}(A))|=|2\sigma_{j}(x_{0})|=|2\sigma_{j}(y_{0})+2|\geq \frac{\N(I)^{\frac{2}{n}}}{4}-2.$
\end{proof}

\vspace{-0.6mm}
With this we can guarantee the Riemannian structure for $M_{I}$:  
\begin{corollary}\label{riemannian structure}
For any ideal $I\subset\mathcal{O}$ with $\N(I)\geq 4^{n}$ the subgroup $\Gamma(I)$ acts on $(\mathbb{H}^{2})^{n}$ freely and so $M_{I}=(\mathbb{H}^{2})^{n}/\Gamma(I)$ admits a structure of a Riemmanian manifold with non-positive sectional curvature.
\end{corollary}

\begin{proof}
The element $A=(\sigma_{1}(A),\ldots,\sigma_{n}(A))\in\Gamma(I)$ has a fixed point on $(\mathbb{H}^{2})^{n}$ if and only if $\sigma_{i}(A)$ has a fixed point in $\mathbb{H}^{2}$ for any $i=1,\ldots.n$, but this happens if only if $\lvert \tr(\sigma_{i}(A))\rvert<2$ which, by Lemma \ref{secondlemma}, is impossible if $\N(I)\geq 4^{n}$.
\end{proof}
Now observe that for $i=1,\ldots,n$ and $A\in\Gamma$, 
\begin{equation}\label{inequalitytrace}
2\lvert  \sigma_{i}(y_{0})\rvert-2 \leq\lvert \tr(\sigma_{i}(A))\rvert\leq 2+2\lvert  \sigma_{i}(y_{0})\rvert.
\end{equation}

\begin{proposition}\label{mainproposition}
Let $I\subset\mathcal{O}$ be an ideal with $\N(I)\geq 40^{\frac{n}{2}}$ and $A\in\Gamma(I)$ with $y_{0}\neq 0$. Then for any point $z=(z_{1},\ldots,z_{n})\in (\mathbb{H}^{2})^{n}$ we have

\begin{equation*}
d_{(\mathbb{H}^{2})^{n}}(z,Az)\geq \frac{4}{\sqrt{n}}\log(\N(I))-2\sqrt{n}\log(40).
\end{equation*}
\end{proposition}
\begin{proof}
By Lemma \ref{secondlemma}, $|\tr(\sigma_{j}(A))|\geq 8$ for some $j\in \{1,\dots,n\}$, hence we can subdivide our analysis into two different cases:\\

\textit{Case 1}: $|\tr(\sigma_{i}(A))|\geq 8$ for any $i=1,\ldots,n.$\\
In this case all of the matrices $\sigma_{i}(A)$ are hyperbolic and the right hand side of Equation \eqref{inequalitytrace} implies that $|\sigma_{i}(y_{0})|\geq 3$ for $i=1,\ldots,n$.\\

Using Equation \eqref{distance and trace}, the left hand side of Equation \eqref{inequalitytrace}, the fact that $|\sigma_{i}(y_{0})|\geq 3$ for $i=1,\ldots,n$, the convexity  of the function $x^{2}$ and Lemma \ref{firstlemma} we obtain

\begin{align*}
d_{(\mathbb{H}^{2})^{n}}(z,Az)&=\sqrt{d_{\mathbb{H}^{2}}^{2}(z_{1},\sigma_{1}(A)z_{1})+\cdots+d_{\mathbb{H}^{2}}^{2}(z_{n},\sigma_{n}(A)z_{n}))}\\
& \geq2\sqrt{\log^{2}(\lvert \tr(\sigma_{1}(A))\rvert-1)+\cdots+\log^{2}(\lvert \tr(\sigma_{n}(A))\rvert-1)}\\
& \geq 2\sqrt{\log^{2}(2\lvert \sigma_{1}(y_{0})\rvert-3)+\cdots+\log^{2}(2\lvert \sigma_{n}(y_{0})\rvert-3)}\\
&\geq 2\sqrt{\log^{2}(\lvert \sigma_{1}( y_{0})\rvert)+\cdots+\log^{2}(\lvert \sigma_{n}(y_{0})\rvert)}\\
&\geq\frac{2}{\sqrt{n}}\left(\log(\lvert \sigma_{1}(y_{0})\rvert)+\cdots+ \log(\lvert \sigma_{n}(y_{0})\rvert)\right)\\
&=\frac{2}{\sqrt{n}}\log(\lvert\N(y_{0})\rvert)\\
&\geq \frac{4}{\sqrt{n}}\log(\N(I))-2\sqrt{n}\log(8).
\end{align*}

\textit{Case 2}: There are exactly $k<n$ of the indices $1,\ldots,n$ such that\\ $|\tr(\sigma_{j}(A))|< 8$.\\
Without loss of generality we assume that $|\tr(\sigma_{j}(A))|< 8$ for $j=~1,\ldots,k$. By the left hand side of Equation \eqref{inequalitytrace}, $|\sigma_{j}(y_{0})|< 5$ for any of these $j's$ and by Lemma \ref{firstlemma} we have

 $$\prod_{i=k+1}^{n}\lvert\sigma_{i}(y_{0})\rvert=\frac{\lvert \N(y_{0})\rvert}{\prod_{i=1}^{k}|\sigma_{i}(y_{0})|}>\frac{1}{5^{n}.8^{n}}\N(I)^{2}.$$

Now, as $|\tr(\sigma_{i}(A))|\geq8$ for $i=k+1,\ldots,n$, for these indices $\sigma_{i}(A)$ is hyperbolic and $|\sigma_{i}(y_{0})|\geq3$ by the left hand side of Equation \eqref{inequalitytrace}. Using Equation \eqref{distance and trace} and the previous facts we have
  
\begin{align*}
d_{(\mathbb{H}^{2})^{n}}(z,Az)&=\sqrt{d_{\mathbb{H}^{2}}^{2}(z_{1},\sigma_{1}(A)z_{1})+\cdots +d_{\mathbb{H}^{2}}^{2}(z_{n},\sigma_{n}(A)z_{n}))}\\
&\geq \sqrt{d_{\mathbb{H}^{2}}^{2}(z_{k+1},\sigma_{k+1}(A)z_{k+1})+\cdots +d_{\mathbb{H}^{2}}^{2}(z_{n},\sigma_{n}(A)z_{n}))}\\
&\geq2\sqrt{\log^{2}(\lvert \tr(\sigma_{k+1}(A))\rvert-1)+\cdots+\log^{2}(\lvert \tr(\sigma_{n}(A))\rvert-1)}\\
& \geq 2\sqrt{\log^{2}(2\lvert \sigma_{k+1}(y_{0})\rvert-3)+\cdots+\log^{2}(2\lvert \sigma_{n}(y_{0})\rvert-3)}\\
&\geq 2\sqrt{\log^{2}(\lvert \sigma_{k+1}( y_{0})\rvert)+\cdots+\log^{2}(\lvert \sigma_{n}(y_{0})\rvert)}\\
&\geq\frac{2}{\sqrt{n-k}}\left(\log(\lvert \sigma_{k+1}(y_{0})\rvert)+\cdots+ \log(\lvert \sigma_{n}(y_{0})\rvert)\right)\\
&=\frac{2}{\sqrt{n-k}}\log\left(\prod_{i=k+1}^{n}\lvert\sigma_{i}(y_{0})\rvert\right)\\
&\geq \frac{4}{\sqrt{n}}\log(\N(I))-2\sqrt{n}\log(40).
\end{align*}
In all the cases we have 
$$d_{(\mathbb{H}^{2})^{n}}(z,Az)\geq \frac{4}{\sqrt{n}}\log(\N(I))-2\sqrt{n}\log(40).$$
\end{proof}

\begin{corollary}\label{maincorollary}
For any ideal $I\subset\mathcal{O}$ with $\N(I)\geq 40^{\frac{n}{2}}$, the length of any non-contractible closed geodesic $\alpha$ in $M_{I}$ satisfies
$$\ell(\alpha)\geq \frac{4}{\sqrt{n}}\log(\N(I))-2\sqrt{n}\log(40).$$

\end{corollary}

\begin{proof}
By Corollary \ref{riemannian structure}, $M_{I}$ is a Riemannian manifold with the metric induced from $(\mathbb{H}^{2})^{n}$. Lifting $\alpha$ to a geodesic $\tilde{\alpha}=(\tilde{\alpha}_{1},\ldots,\tilde{\alpha}_{n})$ in its universal cover $(\mathbb{H}^{2})^{n}$ there is $A\in\Gamma(I)$ acting on $\tilde{\alpha}$ as a translation and for any $z$ in the graph of $\tilde{\alpha}$ we have $\ell(\alpha)=d_{(\mathbb{H}^{2})^{n}}((z,Az))$. Since $\alpha$ is non-contractible, $\tilde{\alpha}$ is not a point, then for some $i\in\lbrace 1,\ldots,n\rbrace$ $\tilde{\alpha}_{i}$ is a non trivial geodesic in $\mathbb{H}^{2}$, and so $\sigma_{i}(A)$ acts on it as a translation. This implies that $\sigma_{i}(A)$ is hyperbolic and, in particular, $\lvert \tr(A)\rvert\neq2$. Since $\lvert \tr(A)\rvert\neq 2$ implies $y_{0}\neq 0$, the result now follows from Proposition~\ref{mainproposition}.
\end{proof}

\section{\textbf{The index $[\Gamma:\Gamma(I)]$ in the prime power case.}}\label{primecase}

In order to relate the systole to the index $[\Gamma:\Gamma(I)]$ we need to estimate this index in terms of $\N(I)$. This result can be deduced from general theory, but we prefer to give a concrete analysis relevant to our problem.\\

If $I=\mathfrak{p}_{1}^{r_{1}}\ldots\mathfrak{p}_{s}^{r_{s}}$ is the decomposition of the ideal $I$ in prime ideals, by the Chinese Remainder Theorem

$$\SL_{2}(\mathcal{O}/I)\cong \SL_{2}(\mathcal{O}/\mathfrak{p}_{1}^{r_{1}})\times\cdots\times \SL_{2}(\mathcal{O}/\mathfrak{p}_{s}^{r_{s}}).$$

Therefore, the computation of $\lvert \SL_{2}(\mathcal{O}/I)\rvert$ reduces to the case $I=\mathfrak{p}^{t}$ where $\mathfrak{p}$ is a prime ideal and $t>0$.

\begin{lemma}
If $I=\mathfrak{p}^{t}$ is a prime power ideal of $\mathcal{O}$ then\\ $\lvert \SL_{2}(\mathcal{O}/I)\rvert< \N(\mathfrak{p})^{3t}.$ 
\end{lemma}
 
\begin{proof}
The quotient ring $\mathcal{O}/\mathfrak{p}^t$ is a local ring with maximal ideal $\mathfrak{p}/\mathfrak{p}^{t}$ $\lvert\mathcal{O}/\mathfrak{p}^{t}\rvert=\N(\mathfrak{p})^{t}$ and $\lvert\mathfrak{p}/\mathfrak{p}^{t}\rvert=\N(\mathfrak{p})^{t-1}.$\\

With the condition $ad-bc=1$ for $a,b,c,d\in\mathcal{O}/\mathfrak{p}^{t}$, we will compute $\lvert \SL_{2}(\mathcal{O}/\mathfrak{p}^{t})\rvert$ directly separating the analysis into two cases:\\

$c\in\mathfrak{p}/\mathfrak{p}^{t}$:\\ 
In this case neither $a$ nor $d$ are in $\mathfrak{p}/\mathfrak{p}^{t}$ hence they are units in $\mathcal{O}/\mathfrak{p}^{t}$, therefore, since $a$, $b$ and $c$ determine $d$ and there are $\N(\mathfrak{p})^{t-1}$ possibilities for $c$, $\N(\mathfrak{p})^{t}$ possibilities for $b$ and $\N(\mathfrak{p})^{t}-\N(\mathfrak{p})^{t-1}$ possibilities for $a$, we have $\N(\mathfrak{p})^{3t-2}(\N(\mathfrak{p})-1)$ matrices in $\SL_{2}(\mathcal{O}/\mathfrak{p}^{t})$ with $c\in\mathfrak{p}/\mathfrak{p}^{t}$.\\

$c\notin\mathfrak{p}/\mathfrak{p}^{t}$:\\
In this case $c$ is a unit and we have $\N(\mathfrak{p})^{t-1}(\N(\mathfrak{p})-1)$ possibilities for $c$; since $a$, $c$ and $d$ determine $b$ and there are $\N(\mathfrak{p})^{t}$ possibilities for $a$ and $d$ , we have  $\N(\mathfrak{p})^{3t-1}(\N(\mathfrak{p})-1)$ elements in $\SL_{2}(\mathcal{O}/\mathfrak{p}^{t})$ with $c\notin\mathfrak{p}/\mathfrak{p}^{t}.$\\

In conclusion, $\lvert \SL_{2}(\mathcal{O}/\mathfrak{
p}^{t})\rvert=\N(\mathfrak{p})^{3t}\left(1-\frac{1}{\N(\mathfrak{p})^{2}}\right) < \N(\mathfrak{p})^{3t}.$

\end{proof}

As a consequence of the Chinese Remainder Theorem and the fact that the norm of ideals is multiplicative, we obtain the following

\begin{corollary}
For any ideal $I\subset\mathcal{O}$, $\lvert \SL_{2}(\mathcal{O}/I)\rvert<\N(I)^{3}$
\end{corollary}

Now, by definition $[\Gamma:\Gamma(I)]\leq\lvert \SL_{2}(\mathcal{O}/I)\rvert$, hence we have proved

\begin{corollary}\label{index and norm}
If $I$ is any ideal of $\mathcal{O}$, then $[\Gamma:\Gamma(I)]< \N(I)^{3}.$ \\
\end{corollary}

\section{\textbf{Proof of the main result}}\label{proofoftheorem} 

Now we can finish the proof of the theorem

\begin{theorem}\label{maintheorem}
Let $K$ be a totally real number field of degree $n$ and $\mathcal{O}$ the ring of integers of $K$. Any sequence of ideals of $\mathcal{O}$ with $\N(I)\rightarrow\infty$ eventually satisfies

$$\sys\pi_{1}(M_{I})\geq \frac{4}{3\sqrt{n}}\log([\Gamma:\Gamma(I)]) -c,$$

where $\Gamma(I)$ is the principal congruence subgroup of $\Gamma=\SL_{2}(\mathcal{O})$ associated to $I$, $M_{I}=(\mathbb{H}^{2})^{n}/\Gamma(I)$
and $c$ is a constant independent of $I$.\\
\end{theorem}

\begin{proof}
For any ideal $I$ with $\N(I)\geq 40^{\frac{n}{2}}$,  Corollary \ref{riemannian structure} implies that $M_{I}$ is a Riemannian manifold with the metric induced by the product metric on $(\mathbb{H}^{2})^{n}$. Now, by Corollary \ref{maincorollary} and Corollary \ref{index and norm}, we conclude that

$$\sys\pi_{1}(M_{I})\geq \frac{4}{3\sqrt{n}}\log([\Gamma:\Gamma(I)])-2\sqrt{n}\log(40).$$

\end{proof}


\bibliographystyle{siam}
\bibliography{mybiblio}

\vspace{4mm}

IMPA. Estrada Dona Castorina 110, 22460-320, Rio de Janeiro, Brazil.\\
\textit{E-mail address}:\hspace{2mm}\texttt{plinio@impa.br}
\end{document}